\documentclass{article}[10pt]

\usepackage{amsmath}%
\usepackage{amssymb}%

\usepackage{geometry}

\newtheorem{theorem}{Theorem}[section]
\newtheorem{corollary}[theorem]{Corollary}
\newtheorem{definition}[theorem]{Definition}

\newtheorem{lemma}[theorem]{Lemma}
\newtheorem{proposition}[theorem]{Proposition}
\newtheorem{remark}[theorem]{Remark}
\newenvironment{proof}[1][Proof]{\textbf{#1.} }{\ \rule{0.5em}{0.5em}}

\newcommand{\refeqn}[1]{(\ref{#1})}

\newcommand{\cinf}[0]{C^{\infty}}
\newcommand{\rank}[0]{\operatorname{rank}}

\begin{document}

\title{A note on symmetries of diffusions within a martingale problem approach}
\author{
Francesco C. De Vecchi\thanks{Dip. di Matematica, Universit\`a degli Studi di Milano, via Saldini 50, Milano, \emph{email: francesco.devecchi@unimi.it}} , Paola Morando\thanks{DISAA, Universit\`a degli Studi di Milano, via Celoria 2, Milano, \emph{email: paola.morando@unimi.it}}  and Stefania Ugolini\thanks{Dip. di Matematica, Universit\`a degli Studi di Milano, via Saldini 50, Milano, \emph{email: stefania.ugolini@unimi.it}}
}
\date{}
\maketitle

\begin{abstract}
A geometric reformulation of the martingale problem
associated with a set of diffusion processes is proposed. This formulation,
based on second order geometry and  It\^o integration on
manifolds,
allows us   to give a
natural and effective definition of  Lie symmetries for diffusion
processes.
\end{abstract}

\section{Introduction}

The theory of infinitesimal symmetries of ordinary and partial differential equations (ODEs and PDEs respectively)  is a classical research
topic in applied mathematics, providing powerful tools both for investigating the qualitative behaviour of  differential equations and for
obtaining  some explicit expression for their solutions (see, e.g., \cite{Olver1993,Stephani1989}). The theory of symmetries of stochastic
differential equations (SDEs) is, in comparison, less developed. There are two main approaches to this problem in the case of
Brownian-motion-driven SDEs. The first approach, based on the Markovian property of solutions to a SDE, looks for the classical Lie symmetries
of the Markov generator, which is an analytical  object (see \cite{DeLara1995,Glover1990,Liao1992}). The second method, directly inspired by Lie
ideas, consists in seeking for some semimartingale transformations   leaving invariant the set of   solutions to the considered SDE (see,
e.g. \cite{Fredericks2007,Gaeta1999,Kozlov2010,Zambrini2004,Srihirun2007,Unal2004}). Both approaches have their own strong and weak points: for
example, the first method permits to treat a larger class of transformations and processes, while  the second one results more convenient in
order
to generalize the deterministic notions of reduction and reconstruction by quadratures (see \cite{DMU2,Cami2009}).\\
In \cite{DMU1} we propose a partial reconciliation of these two programs: in fact, despite working in the second method perspective,  we introduce
a larger class of SDEs transformations which permits both to include all the transformations of the first approach and to obtain all the applications
of the second one.\\
In this article we make a synthesis of the above two approaches from a new prospective. In particular, starting from  the martingale problem
characterization of the
solutions to a SDE, typical of the Markovian setting, we introduce, in the stochastic framework, a geometric formulation of the  symmetry problem. \\
The main idea is to generalize the  well known identification of an ODE on a manifold $M$ with a one-dimensional module $K$ on the tangent bundle of the
zero-jet
space $N=J^0(\mathbb{R},M)=\mathbb{R} \times M$ (or, equivalently,  a
module $K'$ of codimension one on the cotangent bundle $T^*N$). Thanks to this correspondence, the symmetries of an ODE can be identified
 as the family
of diffeomorphisms $\overline{\Phi}:N
\rightarrow N$ transforming the module $K$ (or, equivalently, the module $K'$) into itself. \\
In order to generalize the previous (deterministic) geometric approach to the stochastic framework we need two main ingredients (both introduced by
P.A. Meyer and L. Schwartz
\cite{Meyer1981,Schwartz1982} and thereafter studied by Emery \cite{Emery1989}): second order geometry and It\^o integration on manifolds.
In particular, second order geometry allows us
to introduce diffusors (a generalization of vector fields) and  codiffusors (a generalization of differential forms), while
It\^o integration on  manifolds permits to integrate any codiffusor $\lambda$ along a semimartingale $X$ defined on $M$.\\
In this framework  the usual martingale problem associated with a second order operator $L$ can be reformulated in the following way: a
semimartingale $X$ is a solution to the martingale problem associated with $L$ if and only if, for any $\lambda \in \Lambda_L$, the integral of
$\lambda$ with respect to $X$ is a local martingale, where $\Lambda_L$ is the module of codiffusors annihilating $L$ (see Section
\ref{section_martingale}). In this way we reinterpret the  martingale
problem in terms of a natural geometric object: the module of codiffusors $\Lambda_L$.\\
Therefore,  we prove that a diffeomorphism $\overline{\Phi}$ is a symmetry for the martingale problem associated with $L$, which means
that $\overline{\Phi}$ transforms solutions to the martingale problem into other solutions to the same martingale problem, if and only if the natural
action of the pull-back $\overline{\Phi}^*$ transforms $\Lambda_L$ into itself. The last condition is purely geometrical and  permits to explicitly
compute
the set of symmetries of the martingale problem associated with $L$ (see Section \ref{section_symmetries}).

The paper is organized as follows: in Section \ref{section_preliminaries} we briefly introduce  second order geometry and It\^o integration on
manifolds and in Section \ref{section_transformations} we study the behaviour of the geometric and probabilistic objects introduced in the
previous section with respect to spatial diffeomorphisms and deterministic time changes. With this background, in Section \ref{section_martingale}, we
propose a geometric reformulation of the martingale problem for diffusion processes and in Section \ref{section_symmetries} we exploit it  in order to provide a suitable notion of symmetry and to explicitly compute  the corresponding  determining equations. Finally we
compare the class of Lie symmetries  arising from our approach  with other ones appearing in the literature.

\section{Preliminaries: second order geometry and It\^o integration on manifold}\label{section_preliminaries}

In this section, also in order to fix notations, we briefly recall some basic facts about second order geometry and  It\^o integration on
manifolds. The interested reader is referred to  \cite{Emery1989,Meyer1981,Schwartz1982} for proofs and further details.

\subsection{Second order geometry}

Given a smooth manifold $M$, we denote by $ \cinf (M) $ the set of
real-valued smooth functions defined on $ M $. If $ F $ is a
bundle with base manifold  $ B $, we denote by $ S(F) $ the set of smooth
sections of $ F $.
Finally, if $ M '$ is a manifold and $ n \in \mathbb{N} $, we
denote by $ J^n(M, M') $ the bundle of $ n $ jets of $ n $ times
differentiable functions defined on $ M $ and taking values in $ M
'$.

Let  $ M $ be a  smooth manifold and $u$ be a global coordinate defined on
$\mathbb{R}$. The subset $ u^{- 1} (0) \subset J ^ 2
(M, \mathbb {R}) $ is a submanifold of $ J^2 (M, \mathbb{R})$ and actually a vector subbundle of $ J^2 (M, \mathbb{R})$.

\begin{definition}
The submanifold $ u^{- 1} (0) \subset J ^ 2 (M, \mathbb {R}) $ is
called the \emph{bundle of codiffusors} of the manifold $ M $ and
is denoted by $ \tau^* M $.
\end{definition}

Given a  coordinate system $ x ^ i$   on $ M $, let  $(x ^ i, u, u_ {x^i}, u_
{x ^ ix ^ j}) $ be the associated coordinate system  on $ J ^ 2
(M, \mathbb {R}) $.  Denoting by $ \pi^2: J ^ 2 (M, \mathbb {R}) \rightarrow M $ the
projection of $ J ^ 2 (M, \mathbb{R}) $ into $ M $, we define the
 smooth function $\Pi:J^2(M,\mathbb{R}) \rightarrow \tau^*M $ as
\begin{equation}
\Pi (x ^ i, u, u_{ x^ {i}}, u_{x ^ ix ^ j}) = (x ^ i, 0, u_ {x ^ i}, u_ {x ^ ix ^ j} ).
\end{equation}
The projection $ \pi ^ 2 | _ {\tau ^ * M} $ makes the submanifold
$ \tau ^ * M $ a vector subbundle of $ J ^ 2 (M, \mathbb {R}) $
with base $ M $. The function $ \Pi $ is well defined and is a
morphism  of vector bundles on $ M $.\\
From now on we call \emph{codiffusor} on $M$ a smooth section of the vector bundle $
\tau ^ * M $.\\
Given a smooth function $ f \in \cinf(M)$, let $ D ^ 2f$  denote the
natural lift of $ f $ to $ J ^ 2 (M, \mathbb {R}) $ given in
coordinates  by
\begin{equation}
D^2f(x)=(x^i,f(x),\partial_{x^i}f(x),\partial_{x^ix^j}f(x)).
\end{equation}
\begin{definition}\label{definition_differential}
We call \emph{second differential} of $f \in \cinf(M)$ the codiffusor
$$d^2f:=\Pi(D^2f).$$
Moreover, for $f,g \in \cinf(M)$, we denote by  $df \cdot dg$ the codiffusor
$$df \cdot dg \ := \ \frac{1}{2}(d^2(fg)-gd^2f-fd^2g).$$
\end{definition}

\noindent Since, from the previous definition, we have
\begin{eqnarray*}
d^2x^i&=&(x^i,0,u_{x^j}=\delta^i_j,u_{x^ix^j}=0) \  \\
dx^i \cdot dx^j&=&\left(x^i,0,u_{x^l}=0,u_{x^kx^l}=\frac{1}{2}(\delta_{ki}\delta_{lj}+\delta_{kj}\delta_{li})\right),
\end{eqnarray*}
we can give an explicit coordinate expression for $d^2f$:
\begin{equation}
 d ^ 2f = \partial_{x ^ i}fd ^ 2x ^ i + \partial_{x ^ ix ^ j} fdx ^ i \cdot dx ^ j.
\end{equation}

\begin{remark}\label{remark_codiffusor2}
If $ \lambda $ is a codiffusor on $ M $ and $ x ^ i $ is a
coordinate system on $ M $, there exist unique  functions $
\lambda_i, \lambda_ {ij} = \lambda_ {ji} \in \cinf(M)$ such that
locally
\begin{equation}
\lambda=\lambda_i d^2x^i+\lambda_{ij} dx^i \cdot dx^j.
\end{equation}
\end{remark}

The following theorem provides a useful  characterization of codiffusors on $M$.
\begin{theorem}\label{theorem_codiffusor1}
For any codiffusor $\lambda$ on $M$ there exist $g_i,f_i \in \cinf(M)$, $i=1,...,m$, such that
$\lambda=\sum_{i=1}^{m}g_id^2f_i$.
\end{theorem}

Using Theorem \ref{theorem_codiffusor1} it is possible to extend the product between the differentials of functions given in Definition
\ref{definition_differential} to a product defined for any couple of differential one forms. In a coordinate system $x^i$ the product has
the following representation
$$\omega \cdot \gamma =\omega_i \gamma_j dx^i \cdot dx^j$$
where $\omega= \omega_i dx^i$ and $\gamma=\gamma_j dx^j$.

\begin{definition}
We denote by  $ \tau M $ the dual bundle of $ \tau ^ *  M $ on $M$, and we call it   the \emph{bundle of diffusors}. A
section of $ \tau  M$ is called a \emph{diffusor}.
\end{definition}

Given  a system of coordinates $ x ^ i $ on $ M $, Remark \ref{remark_codiffusor2} ensures that $\{ d ^ 2x ^ i, dx ^ i \cdot dx ^ j \}$ form a
local basis of the fibers of $ \tau ^ * M $. Hence, it is possible to introduce  the local dual basis $\{ \partial_{x^ i}, \partial_{x^ ix ^ j}
\}$ so that
\begin{eqnarray*}
\langle d^2x^i,\partial_{x^j} \rangle &= &\delta^i_j\\
\langle dx^i \cdot dx^j, \partial_{x^k}\rangle &=& 0\\
\langle d^2x^i , \partial_{x^jx^k} \rangle &=&0\\
\langle dx^i \cdot dx^j, \partial_{x^kx^m} \rangle &=& \frac{1}{2}(\delta^i_k\delta^j_m+\delta^i_m \delta^j_k)
\end{eqnarray*}

We remark that the use of symbols $\partial_{x^i}$ and $\partial_{x^ix^j}$ for the basis of $\tau M$ is not misleading, since  the diffusors $\partial_{x^i}$ and $\partial_{x^ix^j}$ are closely related to  the partial derivatives.
Given  a diffusor $ L $ defined on $ M $, it is natural to associate with $L$  the differential operator $ L: \cinf (M) \rightarrow \cinf(M) $ defined by
\begin{equation}
L (f): = \langle d ^ 2f, L \rangle.
\end{equation}
The following result provides a characterization of diffusors through their associated differential operators.

\begin{theorem}\label{theorem_codiffusor2}
Given a diffusor $L$ on $M$, its  associated operator $L: \cinf (M) \rightarrow \cinf(M) $ is a second order linear differential
operator without multiplicative term. Conversely, if $\Lambda: \cinf (M) \rightarrow \cinf(M) $ is a second order linear differential operator  without multiplicative term,
there exists a unique diffusor $L$ on $M$ such that, $\forall f \in \cinf(M)$,
$$\Lambda(f)=L(f).$$
\end{theorem}

Given  two vector fields $ X, Y $   on $ M $, we consider the second order operator $ L_{XY} $ defined by
\begin{equation}
L_{XY}(f)=X(Y(f)).
\end{equation}
By Theorem \ref{theorem_codiffusor2} there exists a diffusor $L_{XY} \in \tau M$ such that, if $ X = X ^ i \partial_ {x ^ i} $ and $ Y = Y ^ i \partial_ {x ^ i} $,
\begin{equation}
L_{XY}=X^iY^j\partial_{x^ix^j}+X^i(\partial_{x^i}Y^j)\partial_{x^j}
\end{equation}
and we have
$$L_{XY}-L_{YX}=[X,Y].$$

\subsection{It\^o integration on manifolds}

Given a filtered probability space $(\Omega, \mathcal F, (\mathcal {F} _t) _ {t \in [0, T]}, \mathbb {P})$, in the following we consider only stochastic
processes (processes for short) adapted with respect
to the filtration $ \mathcal {F} _t $. Moreover, all (local) martingales are always $ \mathcal {F} _t $ (local) martingales.

Given a process $ X $ and a stopping time  $ \tau $,
we denote by $ X ^{\tau} $ the process stopped at $ \tau $. Moreover,  if $X$ and $Z$ are two real continuous semimartingales, their quadratic  covariation is denoted by $[X,Z]$ (although this notation is the same as the above commutator of  vectors fields, the different meaning will be clear from the contest).

\begin{definition}
An almost surely continuous process $X$ taking values in $ M $  is
a \emph{semimartingale} if, $ \forall f \in \cinf (M) $, $ f
(X) $ is a real continuous semimartingale.
\end{definition}
Semimartingales represent the largest class of processes for which It\^o integration can be introduced.
\begin{theorem}\label{theorem_semimartingale3}
Given a  semimartingale $ X $ on $ M $, there exists a
unique linear functional from $ S (\tau ^ * M) $  into the space
of real semimartingales, denoted by
$$\lambda \longmapsto  \int{\langle \lambda,dX_s\rangle}:=\int {\langle \lambda(X_s), dX_s \rangle },$$
such that, for $f \in \cinf(M)$ and $\lambda \in S(\tau^*M)$,
\begin{itemize}
\item $\int{\langle d^2f(X_t),dX_t\rangle}=f(X)-f(X_0);$
\item  $\int{\langle f(X_t) \lambda(X_t),dX_t\rangle}=\int{f(X_t)d\left(\int{\langle \lambda(X_s),dX_s\rangle}\right)_t},$
where the latter integral is the It\^o integral along the real semimartingale $ \int {\langle \lambda (X_s), dX_s \rangle} $.
\end{itemize}
\end{theorem}

\begin{remark}\label{remark_s}
In  Theorem \ref{theorem_semimartingale3} we define a  stochastic integral satisfying 
$$ \left(\int{\langle \lambda, dX_t \rangle} \right)_0=0.$$
However, it is easy  to extend above definition so  that $\left(\int{\langle \lambda, dX_t \rangle }\right)_s=0$ for some $s \in \mathbb{R}$.
\end{remark}

Later on we adopt the notation
$$\int_s^t{\langle \lambda(X_r), dX_r \rangle}=\left(\int{\langle \lambda(X_r), dX_r \rangle }\right)_t-\left(\int{\langle \lambda(X_r), dX_r \rangle} \right)_s.$$

Some useful properties of the  It\^o integral are collected  in the following proposition.

\begin{proposition}\label{proposition_semimartingale2}
Let $X$ be a semimartingale on $M$, $f,g \in
\cinf(M)$, $\lambda,\sigma \in S(\tau^* M)$ and let $\tau$ be a
stopping time. Then
\begin{itemize}
\item $\left(\int\langle \lambda(X_t),dX_t\rangle\right)^{\tau}=\int{\langle    \lambda(X_t^ {\tau}) ,dX_t^{\tau}\rangle}$;
\item $\int{\langle (df \cdot dg)(X_t),dX_t\rangle}=\frac{1}{2}[f(X),g(X)]$.
\end{itemize}
\end{proposition}

\section{Transformations of codiffusors and It\^o integrals}\label{section_transformations}

In this section, in order to introduce a suitable notion of symmetry, we study  the behaviour of codiffusors, diffusors and It\^o integration under spatial transformations and deterministic time changes of the process $X$.\\
Let us fix some preliminary notations:  given a smooth manifold
$M$, we  denote by  $ N = J ^ 0 (\mathbb
{R}, M) = \mathbb {R} \times M $ and we consider the time $ t $  as the first coordinate of $ N $.

\subsection{Transformations of diffusors and codiffusors}

The definition of codiffusors as sections of a suitable subbundle of $J^2(M, \mathbb{R})$ suggests the possibility of generalizing  in a natural way the pull-back of smooth functions and  differential forms to diffusors and codiffusors. The construction is purely geometric and is based on the following theorem.

\begin{theorem} \label{theorem_codiffusor4}
Given  two smooth manifolds $M$ and $M'$ and a smooth map $ \Phi: M \rightarrow M '$, there exists a unique map $ \Phi ^ *: S (\tau ^ * M') \rightarrow S (\tau ^ * M) $
such that , $ \forall f, g \in \cinf (M ') $ and $\forall \lambda, \sigma \in S (\tau ^ * M') $,
\begin{description}
\item[$i)$] $\Phi^*(d^2f)=d^2(\Phi^*(f))$,
\item[$ii)$] $\Phi^*(f\lambda+g\sigma)=\Phi^*(f)\Phi^*(\lambda)+\Phi^*(g)\Phi^*(\sigma).$
 \end{description}
\end{theorem}
\begin{proof}
The uniqueness follows from Theorem \ref{theorem_codiffusor1}.\\
To prove the existence, we set an atlas $ \{U_j \}_{ j \in \mathbb
{N}}$ on $M'$, with local coordinates $ \{y ^ i_j \}$ in $U_j$ on $M'$, and a partition of the unity $ \{\phi_j \}$ subordinated to
$ \{U_j \} $.
If the support $K$  of $ \lambda \in S (\tau ^ * M') $ is contained in the
support  of $\phi_j $ and if
$$\lambda=\lambda_i d^2y^i_j+\lambda_{ik} dy^i_j \cdot dy^k_j,$$
we define
$$\Phi^*(\lambda)=\Phi^*(\lambda_i) d^2\Phi^i_j(x) +\Phi^*(\lambda_{ik}) d\Phi^i_j(x) \cdot d\Phi^k_j(x),$$
where $\Phi^i_j(x)=(y^i_j \circ \Phi)(x)$. Note that, if $ P \in M $ is not in $\Phi^{-1}(K)$, then
$ \Phi^* (\lambda) (P) = 0 $. If $ \lambda $ is any codiffusor, then we
define
$$\Phi^*(\lambda):=\sum_j\Phi^*(\phi_j \lambda).$$
The above relation is well defined as $ \sum_j\Phi^*
(\phi_j \lambda) $ is pointwise a finite sum. Moreover, it is easy to verify
that $ \Phi^* $ satisfies  properties $i)$ and $ii)$.
${}\hfill$\end{proof}


\begin{definition}
Given a diffeomorphism $\Phi: M \to M'$, the map $ \Phi ^ * :  S (\tau ^ * M') \rightarrow S (\tau ^ * M )$ is called the \emph{pull-back} of
codiffusors.  The map $\Phi_*:   (\tau ^ * M) \rightarrow S (\tau ^ * M ')$ defined as $\Phi _ *: = (\Phi ^ {- 1}) ^ *$ is called the \emph{push-forward} of
codiffusors.
\end{definition}

\begin{theorem}\label{theorem_pullback}
Given a diffeomorphism $\Phi:M \rightarrow M'$, there exists
a unique map $\Phi^*:S(\tau M') \rightarrow S(\tau M)$ such that,
 $\forall L \in S(\tau M')$ and $\forall \lambda \in S(\tau^* M)$,
$$\langle \lambda, \Phi ^ * (L) \rangle = \Phi ^ * (\langle \Phi _ * (\lambda), L \rangle).$$
\end{theorem}
\begin{proof}
Given $L \in S(\tau M')$, we  consider  the second order differential
operator $L'$ on $\cinf(M)$ such that $\forall f \in \cinf(M)$
$$L'(f)=\Phi^*L(\Phi_*(f))).$$
By Theorem \ref{theorem_codiffusor2} there exists a unique
diffusor $L' \in S(\tau M)$ such that
$$L'(f)=\langle d^2 f, L'\rangle.$$
Then, by Theorem \ref{theorem_codiffusor1}, we have
$L'=\Phi^*(L)$. ${}\hfill$\end{proof}

\begin{definition}
Given a diffeomorphism $\Phi: M \to M'$, we call $ \Phi^*:S(\tau M') \rightarrow S(\tau M)$ the
\emph{pull-back} of diffusors and $\Phi_*:S(\tau M) \rightarrow
S(\tau M')$ the \emph{push-forward} of diffusors.
\end{definition}

\begin{remark} \label{remark_codiffusor4}
If $ \Phi: M \rightarrow M '$ is a smooth function, $\forall
\mu, \sigma \in S (T ^ * M') $ we have
$$ \Phi ^ * (\mu \cdot \sigma) = \Phi ^ * (\mu) \cdot \Phi ^ * (\sigma). $$
Moreover, if $ \Phi $ is invertible,  $\forall X, Y \in S (TM ') $
$$ \Phi^ * (L_{XY}) = L_{\Phi ^ * (X) \Phi ^ * (Y)}. $$
All the previous expressions hold when we replace $ \Phi ^ * $ with $ \Phi _ * $.
\end{remark}

If we consider a  one-parameter group of diffeomorphisms $\Phi_a$ describing the flow of a vector field $X$, we can give the following definition.

\begin{definition}
Given a vector field  $ X $  on $ M $, with corresponding one-parameter  flow  $ \Phi_a $,  the \emph{Lie
derivative} of a codiffusor (diffusor) $ \lambda $ along $X$  is
\begin{equation}
 \mathcal {L} _ {X} \lambda = \left[ \frac {d}{da} (\Phi_a ^ * \lambda ) \right]_ {a = 0}.
\end{equation}
\end{definition}

The following theorem permits to compute the Lie derivatives of many important objects.

\begin{theorem}\label{theorem_codiffusor5}
Let $X,X_1,X_2$ be three vector fields on $M$, $L$ a diffusor on
$M$, $\lambda$ a codiffusor on $M$,  $f$ a smooth function on $M$,
$\mu,\sigma$ two differential forms and $\Phi:M \rightarrow M'$ a
diffeomorphism from $M$ onto $M'$. Then
\begin{enumerate}
\item $\mathcal{L}_X(f\lambda)=\mathcal{L}_X(f) \lambda+f \mathcal{L}_X(\lambda)$,
\item $\mathcal{L}_X(fL)=\mathcal{L}_X(f) L+f \mathcal{L}_X(L)$,
\item $\mathcal{L}_X(\langle \lambda,L\rangle)=\langle \mathcal{L}_X(\lambda),L\rangle+\langle \lambda,\mathcal{L}_X(L)\rangle$,
\item $\mathcal{L}_X(d^2f)=d^2(\mathcal{L}_X(f))=d^2(X(f))$,
\item $\mathcal{L}_X(L)(f)=X(L(f))-L(X(f))$,
\item $\mathcal{L}_X(\mu \cdot \sigma)=\mathcal{L}_X(\mu) \cdot \sigma+\mu \cdot \mathcal{L}_X(\sigma)$,
\item $\mathcal{L}_X(L_{X_1X_2})=L_{[X,X_1]X_2}+L_{X_1[X,X_2]}$,
\item $\mathcal{L}_{\Phi_*X}(\Phi_* \lambda)=\Phi_*(\mathcal{L}_X(\lambda))$,
\item $\mathcal{L}_{\Phi_*X}(\Phi_*L)=\Phi_*(\mathcal{L}_X(L))$.
\end{enumerate}
\end{theorem}
\begin{proof}
The proof is an easy application of the properties of the
pull-back of diffusors and codiffusors and of the Leibniz rule for the derivative of a
product.\\  ${}\hfill$\end{proof}

From Theorem \ref{theorem_codiffusor5} we  obtain the explicit
coordinate expression of the Lie derivative of diffusors and
codiffusors along a vector field $X$. In particular, if $ X = \phi ^ i \partial_
{x ^ i} $,    $L = A ^ {ij}
\partial_ {x^ ix ^ j} + b ^ i \partial_ {x ^ i} $ and $\lambda=\lambda_i d^2x^i+\lambda_{ij}dx^i\cdot dx^j$, we have
\begin{eqnarray*}
\mathcal{L}_X(L)&=&(\phi^k\partial_{x^k}A^{ij}-A^{ik}\partial_{x^k}\phi^j-A^{kj}\partial_{x^k}\phi^i)\partial_{x^ix^j}\\
&&+(\phi^k\partial_{x^k}b^i-b^k\partial_{x^k}\phi^i-A^{jk}\partial_{x^jx^k}\phi^i)\partial_{x^i}.
\end{eqnarray*}
\begin{eqnarray*}
\mathcal{L}_X\lambda&=&(\phi^k\partial_{x^k}\lambda_{ij}+\lambda_{ik} \partial_{x^j}\phi^k+\lambda_{kj} \partial_{x^i}\phi^k+\lambda_k\partial_{x^ix^j}\phi^k)dx^i \cdot dx^j\\
&&+(\phi^k\partial_{x^k}\lambda_i+\lambda_k \partial_{x^i}\phi^k)d^2x^i.
\end{eqnarray*}
In order to generalize the geometric approach to symmetry problem  from ODEs  to diffusion processes, it is useful to give the following definition.

\begin{definition}\label{definition_module}
A subset  $\Gamma$  of $S(\tau M)$ (or $S(\tau^* M)$) is  a \emph{module of dimension $k$} if
\begin{enumerate}
\item  $\forall L_1,L_2 \in \Gamma$ also $L_1+L_2 \in \Gamma$,
\item  $\forall L \in \Gamma$ and $f \in \cinf(M)$ we have  $f L \in \Gamma$,
\item for each point $P$ there exist a neighborhood $U$ of $P$ and $k$ diffusors (codiffusors) $L_1,...,L_k \in \Gamma$ such that,  $ \forall \, L \in \Gamma$, we have
 $L=\sum_{i=1}^k f_iL_i$ in $U$, where  $f_1,...,f_k$ are suitable functions in $\cinf(M)$. Furthermore, for any $Q \in U$,  $L_1(Q),...,L_k(Q)$ are $k$ linearly independent elements
of $\tau_QM$ (or $\tau^*_QM$).
\end{enumerate}
\end{definition}

In particular, given $ L \in S (\tau M) $ such
that $ L (P) \not = 0 $ for all $P \in M$, we can consider the one-dimensional module
$$\mathfrak{L}_L=\{ f L| f \in \cinf(M)\}$$
and its annihilator, i.e. the set of
codiffusors
$$\Lambda_L=\{\lambda \in S(\tau^* M) | \langle \lambda,L\rangle=0\}$$
which is a module of rank  $(m-1)$ , where $m= \rank(\tau ^ * M)$.

\begin{definition}
Let $\Gamma$ be a $k$-dimensional module on $M$. A diffeomorphism
$\Phi:M \rightarrow M$ is  a \emph{symmetry} of $\Gamma$ if
$\Phi_*(\Gamma)=\Gamma$. A complete vector field $X \in S(T M)$ is
an \emph{infinitesimal symmetry} of $\Gamma$ if the flow $\Phi_a$ of $X$
is a symmetry of $\Gamma$ for all $a \in \mathbb{R}$.
\end{definition}

\begin{theorem}\label{theorem_codiffusor6}
A complete vector field $X$ is a symmetry of $\Gamma$ if and only
if, $\forall L \in \Gamma$, we have  $\mathcal{L}_X L \in
\Gamma$ (or simply $\mathcal{L}_X(\Gamma) \subseteq \Gamma$).
\end{theorem}
\begin{proof}
We give only a sketch of the proof; further details can be found in \cite{DMU1}.\\
If $\Phi_{a,*}(\Gamma)=\Gamma$, evaluating in zero the derivatives with respect to $a$, we get $\mathcal{L}_X(\Gamma) \subseteq \Gamma$. \\
Conversely,  suppose that $\mathcal{L}_X(\Gamma) \subseteq \Gamma$. Let $L_1,...,L_r$ be local generators for $\tau M$ and choose $L_i$
such that, for $i=1,...,k$, they are also local generators for the module $\Gamma$. Given a diffusor $L$,  there exist some functions
$\alpha_1,...,\alpha_k$ and $\beta_1,...,\beta_{r-k}$ (depending on $a$ and $x$) such that $\Phi_{a,*}(L)=\sum_{i=1}^k\alpha_i L_i+\sum_{i=1}^{r-k} \beta_i L_{i+k}$. Since $\mathcal{L}_X(\Gamma) \subseteq \Gamma$, the functions $\alpha_i,\beta_i$ satisfy the following system of first order PDEs
\begin{equation}\label{equation_geo1}
\left(\begin{array}{c}
\partial_a(\alpha)\\
\partial_a(\beta)
\end{array}\right)=
\left(\begin{array}{c}
X(\alpha)\\
X(\beta)
\end{array}\right)+\left( \begin{array}{cc}
A(x) & B(x) \\
0 & C(x)
\end{array}\right) \cdot \left( \begin{array}{c}
\alpha\\
\beta
\end{array}
\right),
\end{equation}
where $\alpha=(\alpha_1,...,\alpha_k)$, $\beta=(\beta_1,...,\beta_{r-k})$ and $A(x),B(x),C(x)$ are suitable matrix-valued functions. Using the method of characteristics it is  possible to prove that equation \refeqn{equation_geo1} admits a unique solution for any smooth initial value $\alpha(0),\beta(0)$. Moreover,   the form of equation \refeqn{equation_geo1} ensures  that, if $\beta(0)=0$,  then $\beta(a)=0$ for any $a \in \mathbb{R}$. Hence, since  $L \in \Gamma$, we have that  $\beta(0)=0$ and $\Phi_{a,*}(L) \in \Gamma$ for any $a$. ${}\hfill$\end{proof}

\subsection{It\^o integral and  space and time transformations}\label{subsection}

In the following we study the behaviour of a semimartingale under space and time  transformations.

\begin{proposition}\label{proposition_seimaringale1}
Given two manifolds  $M$ and $M'$, a semimartingale $ X $   on $ M $ and a smooth function $ \Phi:
M \rightarrow M '$, the process $ X '= \Phi (X) $,
defined as $ X'_t = \Phi (X_t) $, is a semimartingale on $ M' $.
\end{proposition}
\begin{proof}
The proof is an easy consequence of the  definition of semimartingale on a manifold.${}\hfill$
\end{proof}

\medskip
In order to introduce  time transformations, we consider  a strictly increasing function $ f \in \cinf
(\mathbb {R}) $, so that also $ f ^ {- 1} $  is a
smooth strictly increasing function. If $ X $ is a semimartingale
on $ M $, we  denote by $ X '= H_f (X) $ the process
$$ X '_ {t'} = H_f (X) _ {t '}:= X_ {f ^ {- 1} (t')}. $$
Moreover, working towards a unified description of space and time transformations, we consider a smooth map $\Phi:M \rightarrow M'$,  a
deterministic time change $f$ and a semimartingale $X$ on $M$, and we define
$$\Phi_f(X)=H_f(\Phi(X)).$$

\begin{theorem}\label{theorem_semimartingale5}
With the above notations, $\forall \lambda \in S(\tau^*M')$
$$\int{\langle \lambda(\Phi_f(X)_t),d\Phi_f(X)_t\rangle}=H_f\left(\int{\langle \Phi^*(\lambda)(X_t),dX_t\rangle}\right).$$
\end{theorem}
\begin{proof}
We define the linear operator $I$ from $S(\tau^*M')$ into the set of real semimartingales such that
$$I(\lambda)=H_f\left(\int{\langle \Phi^*(\lambda)(X_t), dX_t \rangle} \right).$$
Using Theorem \ref{theorem_codiffusor4}, Theorem \ref{theorem_semimartingale3} and the definition of $\Phi_f$, we have
\begin{eqnarray*}
I(d^2g)&=&H_f\left(\int{\langle \Phi^*(d^2g)(X_t), dX_t \rangle} \right)\\
&=&H_f\left(\int{\langle d^2\Phi^*(g)(X_t), dX_t \rangle} \right)\\
&=&H_f(g(\Phi(X))-g(\Phi(X)_0))=g(\Phi_f(X))-g(\Phi_f(X)_0),
\end{eqnarray*}
Furthermore, the change rule of It\^o integral with respect to absolutely continuous time changes (see, e.g., \cite[Proposition
30.10]{Williams1987}) ensures that
\begin{eqnarray*}
I(g\lambda)&=&H_f\left(\int{\langle \Phi^*(g\lambda)(X_t), dX_t \rangle} \right)\\
&=&H_f\left(\int{\langle \Phi^*(g)(X_t)\Phi^*(\lambda)(X_t), dX_t \rangle} \right)\\
&=&H_f\left(\int{\Phi^*(g)(X_t) d\left(\int{\langle \Phi^*(\lambda), dX_s \rangle}\right)_t}\right)\\
&=&\int{g(\Phi_f(X)_t)dI(\lambda)_t}.
\end{eqnarray*}
Hence, using the characterization of It\^o integral given in Theorem \ref{theorem_semimartingale3}, we have $I(\lambda)=\int{\langle \lambda,d\Phi_f(X) \rangle}$
and this completes the proof.${}\hfill$
\end{proof}

\medskip

Given a semimartingale  $X$ on $M$, the semimartingale
$$ \overline {X} _t = (t, X_t) \in N $$
is called the \textit{lifting} of $ X $ to $ N $. When $ \lambda \in S (\tau ^ * N) $, we  use the following   notation
$$\int{\langle \lambda,dX_t\rangle}:=\int{\langle \lambda(\overline{X}_t),d\overline{X}_t\rangle}.$$

Given a transformation $ \overline {\Phi} : N \rightarrow N'$, we write $ \overline {\Phi}= (f, \Phi)$, where  $ f$  is the component of $ \overline {\Phi} $ over $
\mathbb {R} $  and $ \Phi$ is  the component of $ \overline {\Phi} $ over
 $ M $. If  $ f $ depends only on $ t$ we say that $ \overline {\Phi}= (f, \Phi)$ is \emph{projectable}.
We call \emph{semimartingale transformation} any  diffeomorphism  $ \overline {\Phi}$ which is projectable. We denote by $ X' = \Phi_f (\overline{X}) $ the transformed semimartingale given by
$$X'_{t'}=\Phi_f(f^{-1}(t'),X_{f^{-1}(t')}).$$

\begin{remark}\label{remark_codiffusor1}
The lifting  $\overline{X}'$ of $ X' $ to $ N '$ satisfies
$$\overline{X'}=H_f(\overline{\Phi}(\overline{X})).$$
\end{remark}

\begin{theorem}\label{theorem_semimartingale6}
Let $ \overline {\Phi} = (f, \Phi): N \rightarrow N '$ be a
semimartingale transformation, and $\lambda \in S (\tau ^
* N ') $; then
$$\int{\langle \lambda, d\Phi_f(X)_t\rangle}=H_f\left(\int{\langle \overline{\Phi}^*(\lambda),dX_t\rangle}\right).$$
\end{theorem}
\begin{proof}
The proof is a simple application of Theorem \ref{theorem_semimartingale5} and Remark \ref{remark_codiffusor1}. ${}\hfill$
\end{proof}

\section{A novel formulation of the martingale problem via second order geometry}\label{section_martingale}

It is  well-known that the martingale problem approach, due to Stroock and Varadhan (\cite{Stroock1979}), represents a modern and fruitful way to introduce diffusion processes, alternative to the classical definition as SDEs solutions.
For a complete exposition of the topic see \cite{Watanabe1981,Stroock1979}  (\cite{Elworthy1982,Hsu2002} for the manifolds setting). In the following we call $X$ a semimartingale starting at time $s \in \mathbb{R}$ if $X_{s+t}$ is a semimartingale (starting at time 0).

\begin{definition}
A semimartingale $ D $ on $ M $ starting at time $s$ is a solution to the martingale
problem associated with a diffusor $ L $ up to a stopping time $\tau>s$ if, $\forall g \in \cinf
(N) $, the real semimartingale $ D^g $ given by
$$D^g_t=g(\overline{D}_{t \wedge \tau})-g(\overline{D}_s)-\int_s^{t \wedge \tau}{L(g)(\overline{D}_{r})dr},$$
is a local martingale (starting at s). A semimartingale solution to the martingale problem associated with a diffusor $ L $ is called a
\emph{diffusion process} (or simply a diffusion) of diffusor $L$.
\end{definition}

When not strictly necessary, we omit the stopping time $\tau$ from the definition of solution to a martingale problem.
Furthermore, unless otherwise stated,  we consider  the solution to the martingale problem starting at $0$.\\

The diffusor $ L  $ is \emph{standard} if, whenever
$ g \in \cinf (N) $ depends only on  $ t $,
$${L}(g)(t)=\frac{dg}{dt}(t).$$

\begin{remark}\label{remark1}
If $ X $ is a continuous local martingale of bounded variation such that $ X_0 = 0 $, then, by martingale property,  $X_t = 0 $ for every $ t \in \mathbb{R}_+ $
 (see, e.g. \cite{Williams1987}).
\end{remark}

The next result shows that our definition of standard diffusor is a natural requirement.
\begin{proposition}\label{proposition_symmetry1}
A diffusor $L$ is standard if there exists a diffusion $ D $ of diffusor $ L $.
\end{proposition}
\begin{proof}
If  $ g \in \cinf (N) $ depends only on $t$, considering $\overline{D}_t=(t, D_t)$, we have that
$$D^g_t=g(\overline{D}_t)-g(\overline{D}_0)-\int_0^t{L(g)(\overline{D}_s)ds}
=g(t)-g(0)-\int_0^t{L(g)(\overline{D}_s)ds},$$
is a local martingale.\\
Being $ g (t) -g (0) $  and $ \int {L (g) (\overline {D} _s) ds} $  bounded variation processes,  $ D ^ g $ is a bounded variation local martingale and,  by Remark \ref{remark1}, $ D ^ g_t = 0 $, which implies that
$ g (t) -g (0) = \int {L (g) (\overline {D} _s) ds} $. By differentiating  both sides of the latter equality with respect to $t$, we get
$$\frac{dg}{dt}(t)=L(g)(\overline{D}_t),$$
which means that $L (g) = dg / dt $, i.e. $ L $ is standard.
${}\hfill$\end{proof}

\medskip

In the following we  associate with each martingale problem a well-defined module of codiffusors and
we prove that this  module is actually completely equivalent to the martingale problem.
We start with a preliminary lemma.

\begin{lemma}\label{lemma_symmetry2}
Let $L$ be a standard diffusor and $ D  $ be a diffusion of diffusor $L$. For any  $\mu \in S (\tau ^ * N) $ we consider the codiffusor  $ \lambda = \mu - \langle \mu, L \rangle d^2t $. Then $ \int {\langle \lambda , dD_t \rangle} $ is a local martingale.
\end{lemma}

\begin{proof}
If $\mu=d^2h$ with  $ h \in \cinf (N) $  the lemma reduces to the definition of a diffusion of diffusor $ L $. \\
If $ \mu $ is a generic codiffusor, by Theorem \ref{theorem_codiffusor1} there exist  $ f_i, g_i \in \cinf (N) $ such that
$$\mu=\sum_i g_i d^2f_i.$$
If we consider $ \lambda_i = d ^ 2f_i-L (f_i)  d^2t$, we have that $ \int {\langle \lambda_i, dD_t \rangle} $ is a local martingale and, being
$$\lambda=\mu-\langle \mu, L\rangle d^2t=\sum_i g_i\lambda_i,$$
we find
\begin{eqnarray*}
\int{\langle \lambda, dD_t\rangle}&=&\sum_i\int{\langle g_i \lambda_i,dD_t\rangle}
\ = \ \sum_i\int{g_i(\overline{D}_t)d\left(\int{\langle \lambda_i,dD_s\rangle}\right)_t}.
\end{eqnarray*}
The latter integral is a local martingale, being a sum of It\^o integrals along the real local martingales $ \int {\langle \lambda_i, dD_s \rangle} $.
${}\hfill$\end{proof}

\medskip

We recall that $ \Lambda_L \subset S (\tau ^ * N) $ denotes the \emph{annihilator} of the one-dimensional module $\mathfrak{L}_L$ generated by the diffusor $ L \in S (\tau N) $.

\begin{theorem}\label{theorem_symmetry1}
The semimartingale $ D $ on $ M $ is a diffusion of standard diffusor $ L $, if and only if, for every $ \lambda \in \Lambda_L $,
$$\int{\langle \lambda, dD_t\rangle}$$
is a local martingale.
\end{theorem}
\begin{proof}
By Lemma \ref{lemma_symmetry2}, if $ D $ is a diffusion of standard diffusor $ L $ and $ \lambda \in \Lambda_L $, then $ \int {\langle \lambda, dD_t \rangle} $ is a local martingale. Indeed we know that $ \lambda -\langle \lambda, L \rangle  d ^ 2t$ integrated along $ D $ is a local martingale and that $ \langle \lambda, L \rangle = 0 $ because $ \lambda \in \Lambda_L $.\\
Conversely, suppose  that the semimartingale $ D $ is such that, $\forall \lambda \in \Lambda_L $,  $ \int {\langle \lambda, dD_t \rangle} $ is a local martingale.
Given  $ g \in \cinf (N) $, we have $ \lambda = d ^ 2g-L (g) d ^ 2t \in \Lambda_L $,  being $\langle d^2t, L \rangle =1$ and $\langle d^2g, L \rangle =L(g)$. Hence
$$\int \langle (d^2 g-L(g)d^2t),dD_t \rangle = g(\overline{D})-g(\overline{D}_0)-\int L(g)(\overline{D}_t)dt$$
is a local martingale. Since $g$ is a generic function in $\cinf(N)$, then $D$ is a diffusion of diffusor $L$.\\
${}\hfill$\end{proof}

In Lemma \ref{lemma_symmetry2} and Theorem \ref{theorem_symmetry1} we have implicitly assumed that the stopping time $\tau$ of the diffusion $D$
is equal to $+ \infty$. The general case can be recovered by using Proposition \ref{proposition_semimartingale2}.\\
In order to  prove a sort of converse of Theorem \ref{theorem_symmetry1}, since we do not need the uniqueness of the solution to the martingale problem, instead of the well-posedness notion we introduce  the following definition.\\

\begin{definition}\label{defintion_usual}
A diffusor $L$ is a \emph{good} diffusor if, for any $t_0 \in \mathbb{R}$ and $x_0 \in M$,  there exists at least one diffusion $D$, starting
at $t_0$ and such that $D_{t_0}=x_0$ almost surely, solution to the martingale problem associated with $L$.
\end{definition}



\begin{proposition}\label{proposition_symmetry2}
If $L$ is  a good diffusor,  then
$$\Lambda^{\prime}:=\left\{ \lambda \in S(\tau^*N) \left| \int{\langle \lambda,dD_t\rangle} \text{ is a local martingale } \right. \right\}\subseteq \Lambda_L$$
\end{proposition}
\begin{proof}
Given $ \lambda \in \Lambda '$, by Lemma \ref{lemma_symmetry2}, $ \lambda - \langle \lambda, L \rangle  d^2t \in \Lambda '$ and, being $ \Lambda '$ closed with respect to the sum, we have
$$\langle \lambda,L\rangle d^2t=\lambda-(\lambda-\langle \lambda,L\rangle d^2t) \in \Lambda'.$$
Let $D^{x_0,t_0}$ be a diffusion starting at $t_0$ such that $D^{x_0,t_0}_{t_0}=x_0$. The integral
$$\int{\langle (\langle \lambda,L\rangle d^2t),dD^{x_0,t_0}_t\rangle}=\int{(\langle \lambda,L\rangle)(\overline{D}^{x_0,t_0}_t)dt}$$
is a  bounded variation process and also a local martingale and,
by Remark \ref{remark1},
$$
\int_{t_0}^r{(\langle \lambda,L\rangle)(\overline{D}^{x_0,t_0}_t)dt}=0
$$
for any $r > t_0$.
Since $(\langle \lambda,L\rangle)(\overline{D}^{x_0,t_0}_t)$ is continuous with respect to $t$, we  have that $(\langle \lambda,L\rangle)(\overline{D}^{x_0,t_0}_t)=0$ and, considering the limit for $t \rightarrow t_0$ in the previous expression, we get $(\langle \lambda,L\rangle)(t_0,x_0)=0$. Since $x_0 \in M$ and $t_0 \in \mathbb{R}$ are generic points the proposition is proved.
${}\hfill$\end{proof}

\begin{corollary}
If $L$ is a good diffusor, then
$$\Lambda_L=\left\{ \lambda \in S(\tau^*N) \left| \int{\langle \lambda,dD_t\rangle} \text{ is a local martingale } \right. \right\}.$$
\end{corollary}

In the following  we always consider  good diffusors $L$. This choice is not restrictive since, using the stopping time $\tau$ and our definition of solution
to the martingale problem, we can exploit all existence results for diffusion processes in $\mathbb{R}^n$ (see \cite{Stroock1979}).

\section{Symmetries of diffusions}\label{section_symmetries}

Generalizing the natural idea of symmetries of ODEs as diffeomorphisms transforming solutions into solutions, we  give the following definition.

\begin{definition}
Let $ \overline {\Phi}: N \rightarrow N $ be an invertible semimartingale transformation. The diffeomorphism $ \overline {\Phi} = (f, \Phi) $ is
a symmetry of the diffusions associated with $ L $ (in short, a symmetry of $L$) if, for any diffusion $ D $ of diffusor $ L $,  also $ \Phi_f (D) $
is a diffusion of diffusor $ L $.
\end{definition}
The next result characterizes  symmetries of diffusions associated with a diffusor $L$ in terms of a suitable  invariance property of the module of codiffusors $\Lambda_L$.
\begin{theorem}\label{theorem_symmetry2}
An invertible semimartingale transformation $ \overline {\Phi}: N \rightarrow N $ is a  symmetry of $ L $ if and only if $\overline{\Phi}$ is a symmetry of $\Lambda_L$.
\end{theorem}
\begin{proof}
Suppose that $ \overline {\Phi} $ is a  symmetry of $ L $, and let $D$ be any diffusion of diffusor $L$. Obviously, by Theorem \ref{theorem_symmetry1} and by the definition of  symmetry, $\forall \lambda \in \Lambda_L$
$$\int{\langle \lambda, d\Phi_f(D)_t\rangle}$$
is a local martingale. On the other hand, by Theorem \ref{theorem_semimartingale6} we have
$$\int{\langle \lambda, d\Phi_f(D)_t\rangle}=H_f\left(\int{\langle \overline{\Phi}^*(\lambda), dD_t\rangle}\right).$$
Since the latter equality and  Proposition \ref{proposition_symmetry2} ensure  that $\overline{\Phi}^*(\lambda) \in \Lambda_L$, then   $\overline{\Phi}^*(\Lambda_L) \subseteq \Lambda_L$. The equality follows from the invertibility of $\overline{\Phi}$.\\
Conversely, suppose that $\overline{\Phi}^*(\Lambda_L)=\Lambda_L$ and let  $D$ be any diffusion of diffusor $L$. Fixing $\lambda \in \Lambda_L$, from Theorem \ref{theorem_semimartingale6} we have
$$\int{\langle \lambda, d\Phi_f(D)_t\rangle}=H_f\left(\int{\langle \overline{\Phi}^*(\lambda),dD_t\rangle}\right).$$
Since $\overline{\Phi}^*(\lambda) \in \Lambda_L$, the right-hand side of the last equality is a local martingale. Then, by Theorem \ref{theorem_symmetry1}, $\Phi_f(D)$ is a diffusion of diffusor $L$.
${}\hfill$\end{proof}

In order to provide a simpler characterization of  symmetries of $L$ we give the following lemma.

\begin{lemma}\label{lemma_symmetry3}
Let $ L $ be a standard diffusor. If there exists a diffusor  $ L '$ such that, $\forall \lambda \in \Lambda_L $,
$ \langle \lambda, L '\rangle = 0 $,
then there exists $ \mu \in \cinf (N) $ such that $ L '= \mu L $.
\end{lemma}

\begin{proof}
Let us consider  $\widetilde{L}=L'-L'(t) L$:  we show that $\widetilde{L}=0$ proving  that, $\forall  g \in \cinf(N)$,  $\widetilde L(g)=0$.
Since $L$ is standard,  $L(t)=\langle d^2t, L\rangle =1$ and
\begin{eqnarray*}
\widetilde{L}(t)&=&\langle d^2t,L'\rangle- L'(t) \langle d^2t, L\rangle\\
&=&L'(t)-L'(t)L(t)=0.
\end{eqnarray*}
Obviously, if $\lambda \in \Lambda_L$, then $\langle \lambda , \widetilde{L}\rangle =0$. So, if $g \in \cinf(M)$, then $\lambda = d^2g- L(g)d^2t \in \Lambda_L$. Therefore
\begin{eqnarray*}
\widetilde{L}(g)&=&\langle (\lambda+L(g)d^2t),\widetilde{L}\rangle\\
&=&\langle \lambda ,\widetilde{L}\rangle +L(g)\widetilde{L}(t)=0, \end{eqnarray*}
and, by  Theorem \ref{theorem_codiffusor2},  the last statement is equivalent to $\widetilde{L}=0$.
${}\hfill$\end{proof}

With the notations and the hypothesis of  Theorem \ref{theorem_symmetry2} , $ \overline {\Phi} $ is a symmetry of $L$ if and only if
$ \overline {\Phi} ^ * (L) = \mu L $, for some $ \mu \in \cinf (M) $ such that $ \mu \not = 0 $.

\begin{definition}\label{definition_symmetry1}
Let $X$ be a complete vector field on $N$   with corresponding flow $\overline{\Phi}_a$.The vector field $X$  is  an \emph{infinitesimal symmetry}
for the diffusions  associated with a diffusor $L$ (in short an infinitesimal symmetry for $L$) if, $\forall a \in \mathbb{R}$,  $\overline{\Phi}_a$ is a
symmetry of the diffusor $L$.
\end{definition}

\begin{remark}\label{remark_symmetry1}
A necessary condition for $X$ to be an infinitesimal symmetry of a diffusor $L$  is that the flow $ \overline {\Phi} _a $ is a one-parameter group of invertible semimartingale transformations.
This is equivalent to require that $X$ is \emph{projectable}, i.e. the vector field $X$ is of the form $X=\phi^i \partial_{x^i}+\tau \partial_t$, where the function $\tau$ depends only on $t$.
\end{remark}

\begin{theorem}\label{theorem_symmetry4}
A projectable complete vector field $X$ is an infinitesimal symmetry of a standard diffusor  $ L $  if and only if $X$ is a symmetry of $\Lambda_L$, i.e.
\begin{equation}\label{equation_lambda}
\mathcal{L}_X(\Lambda_L) \subseteq \Lambda_L.
\end{equation}
\end{theorem}
\begin{proof}
The necessity (and the sufficiency) of the existence of the flow and of the projectability of $X$ are explained in   Remark \ref{remark_symmetry1}.\\
Besides, since $\Lambda_L$ is a $k$-dimensional  module (with $k=\rank(\tau^*N)-1$) the necessity and sufficiency of  condition \refeqn{equation_lambda} are  simple consequences of Theorem \ref{theorem_codiffusor6} and Theorem \ref{theorem_symmetry2}.
${}\hfill$
\end{proof}

The following proposition provides a very useful  condition, ensuring that a complete vector field is a symmetry of a diffusion  $L$.
\begin{proposition}\label{proposition_symmetry4}
Let $X$ be a projectable complete vector field and $L$ be a standard diffusor. Then $ \mathcal {L} _X (\Lambda_L) \subseteq \Lambda_L $ if and only if there exists
$ \mu \in \cinf (N) $ such that
\begin{equation}\label{equation_lambda2}
\mathcal {L} _X (L) = \mu L.
\end{equation}
\end{proposition}
\begin{proof}
Suppose that $\mathcal{L}_X(\Lambda_L) \subseteq \Lambda_L$. For any codiffusor  $\lambda \in \Lambda_L$, we have
\begin{eqnarray*}
0&=&\mathcal{L}_X(\langle \lambda, L\rangle)\\
&=&\langle \mathcal{L}_X(\lambda),L \rangle +\langle \lambda, \mathcal{L}_X(L)\rangle\\
&=&\langle \lambda, \mathcal{L}_X(L)\rangle.
\end{eqnarray*}
Hence, by Lemma \ref{lemma_symmetry3}, there exists  $\mu \in \cinf(N)$ such that $\mathcal{L}_X(L)=\mu L$.\\
Conversely, suppose that $\mathcal{L}_X(L)=\mu L$; then for any $\lambda \in \Lambda_L$,
\begin{eqnarray*}
0&=&\langle \mathcal{L}_X(\lambda), L\rangle+\langle \lambda, \mu L\rangle\\
&=&\langle \mathcal{L}_X(\lambda), L\rangle.
\end{eqnarray*}
Hence $\mathcal{L}_X(\lambda) \in \Lambda_L$, completing the proof.
${}\hfill$\end{proof}

In order to give a coordinate expression for  condition  \refeqn{equation_lambda2} we consider  a coordinate system $x^i$  on $M$ and  a standard  diffusor $L$  of the form
$$L=A^{ij}\partial_{x^ix^j}+b^i\partial_{x^i}+A^{it}\partial_{x^it}+\partial_t.$$
It is easy to prove that, if $L$ is a good diffusor, then  $A^{it}=0$ and  the matrix $A^{ij}$ is semidefinite positive. Hence
 $L$ has the form
\begin{equation}\label{equation_symmetry1}
L=A^{ij}\partial_{x^ix^j}+b^i\partial_{x^i}+\partial_t.
\end{equation}
Given a projectable vector field  $X=\phi^i\partial_{x^i}+\tau \partial_t$, we can calculate $\mathcal{L}_X(L)$ and, inserting this expression in \eqref{equation_lambda2},  we obtain $\mu=-\partial_t\tau$ and
\begin{eqnarray}
&\phi^k\partial_{x^k}A^{ij}+\tau\partial_tA^{ij}-A^{ik}\partial_{x^k}\phi^j-A^{kj}\partial_{x^k}\phi^i+A^{ij}\partial_t\tau=0&\label{equation_symmetry2}\\
&\phi^k\partial_{x^k}b^i+\tau \partial_tb^i-b^k\partial_{x^k}\phi^i-A^{jk}\partial_{x^jx^k}\phi^i+b^i\partial_t\tau-\partial_t\phi^i=0,&\label{equation_symmetry3}
\end{eqnarray}
for $i,j=1,...,\dim(M)$.

\bigskip

In the following we compare the symmetry approach proposed in this paper, and in particular the \emph{ determining equations} \refeqn{equation_symmetry2}
and \refeqn{equation_symmetry3}, with  other results on symmetries of stochastic processes appearing  in the literature.\\
Given a diffusor $L$, it is natural to consider the corresponding Kolmogorov equation
\begin{equation}\label{equation_Kolmogorov}
L(u)=A^{ij}\partial_{x^ix^j}(u)+b^i\partial_{x^i}(u)+\partial_t(u)=0
\end{equation}
describing the behaviour of the mean value of regular functions of the solution process $X_t$. More precisely, a solution $u(x,t)$ to
equation \refeqn{equation_Kolmogorov} is of the form $\mathbb{E}[f(X_T)|X_t=x]=u(x,t)$, with $t\in [0,T]$. Since
\refeqn{equation_Kolmogorov} is a PDE, its Lie symmetries can be interpreted as vector fields on $J^0(N,\mathbb{R})$ of the form
$$
Z=\tau(x,t,u)\partial_{t}+\phi^i(x,t,u)\partial_{x^i}+\psi(x,t,u)\partial_u.
$$
satisfying (in the non-degenerate case, i.e. when $A^{ij}$ has maximal rank) the following conditions (see e.g. \cite{Gaeta1999,Olver1993})
\begin{eqnarray*}
&\psi(x,t,u)=h(x,t)u&\\
&\partial_u(\phi^i)=0&\\
&\partial_u(\tau)=0&\\
&\partial_{x^i}(\tau)=0&\\
&\partial_t(h)+A^{ij}\partial_{x^ix^j}(h)+b^i\partial_{x^i}(h)=0&\\
&\phi^k\partial_{x^k}A^{ij}+\tau\partial_tA^{ij}-A^{ik}\partial_{x^k}\phi^j-A^{kj}\partial_{x^k}\phi^i+A^{ij}\partial_t\tau=0&\\
&\phi^k\partial_{x^k}b^i+\tau \partial_tb^i-b^k\partial_{x^k}\phi^i-A^{jk}\partial_{x^jx^k}\phi^i+b^i\partial_t\tau-\partial_t\phi^i+A^{ik}\partial_{x^k}(h)+A^{ki}\partial_{x^k}(h)=0.&
\end{eqnarray*}

It is interesting to note  that  these equations coincide with equations \refeqn{equation_symmetry2} and \refeqn{equation_symmetry3} when $h$
is constant. This is due to the fact that, in our approach, the main object is the process $X_t$ and a symmetry $Y$ on $M$
 transforms the solution $X_t$ to the martingale problem into a (possibly different) solution $\Phi_f(X)_t$ to the same martingale problem. Hence $Y=Z$ (under the hypothesis $h=0$) transforms  solutions to \refeqn{equation_Kolmogorov} into  other solutions to \refeqn{equation_Kolmogorov}.\\
Indeed a solution $u$ to the Kolmogorov equation such that $u(x,T)=g(x)$ is of the form $u(x,t)=\mathbb{E}[g(X_T)|X_t=x]$. This means that
\begin{eqnarray*}
u(\Phi_a^{-1}(x,t),f_a^{-1}(t))&=&\mathbb{E}[g(X_T)|X_{f_a^{-1}(t)}=\Phi_a^{-1}(x)]\\
&=&\mathbb{E}[g(X_T)|\Phi_a(H_{f_a}(\overline{X}))=x]\\
&=&\mathbb{E}[g \circ \Phi_a^{-1}(\Phi_{a,f_a}(X)_{f_a^{-1}(T)},f_a^{-1}(T)) \, |\Phi_{a,f_a}(X)_t=x]=v(x,t).
\end{eqnarray*}
Since $\Phi_f(X)$ is still a solution to the same  martingale problem, we have that $v$ is the unique solution to $L(v)=0$ with final condition
$v(x,f_a^{-1}(T))=g(\Phi_a^{-1}(x,f_a^{-1}(T)))$. \\
The fact that only the transformations with $h=0$ turn out to be symmetries of both the diffusion process and the Kolmogorov equation follows from the fact that
the transformations of the function $u$ do not  have a natural meaning when the focus is on the process.  \\

Another  natural comparison arising in this framework is the study of the relationship  between  the symmetries of a martingale problem as proposed in the present paper and  the symmetries of the corresponding  SDE as given in \cite{DMU1}. Since in \cite{DMU1} we consider only autonomous SDEs and stochastic  time changes, in order to make the two approaches correctly comparable, we restrict our considerations to autonomous diffusions (i.e. $A^{ij},b^i$  not depending on $t$) and  time changes of the form  $\tau=a t$ for some $a \in \mathbb{R}$.\\
Given  $\mu:M \rightarrow \mathbb{R}^n$ and $\sigma:M \rightarrow Mat(n,m)$,  in \cite{DMU1} we consider  SDEs of the form
$$dX_t=\mu^i(X_t) dt +\sigma^i_{\alpha}(X_t) dW^{\alpha}$$
where  $(\mu,\sigma)=(\mu^i(x),\sigma^i_{\alpha}(x))$,  and  $m$ is the dimension of the Brownian motion driving the SDE. The relationship between $(\mu,\sigma)$ and $(b^i,A^{ij})$ is provided  by It\^o formula which  ensures that
\begin{equation}\label{equation_coeffSDE}
b^i=\mu^i,  \ \ A^{ij}=\frac{1}{2}\sum_{\alpha=1}^m \sigma^i_{\alpha}\sigma^j_{\alpha}.
\end{equation}
The infinitesimal stochastic transformation of a  process $X$  and of a Brownian motion $W$ is given by a triple $(\tilde{Y},C,a)$
where $\tilde{Y}=\phi^i\partial_{x^i}$ is a vector field on $M$,  describing the spatial change of $X$, $C:M \rightarrow \mathfrak{so}(m)$ is
a function representing  the random rotation of the Brownian motion $W$ and taking values in the group of antisymmetric matrices,  and $a \in \mathbb{R}$ is the parameter of the time rescaling.\\
The determining equations for $(\tilde{Y},C,a)$ are
\begin{eqnarray}
&\phi^k\partial_{x^k}\mu^i-\mu^k\partial_{x^k}\phi^i-\frac{1}{2}\sum_{\alpha}\sigma_{\alpha}^{j}\sigma_{\alpha}^k\partial_{x^jx^k}\phi^i+a \mu^i=0,&\label{equation_determining1}\\
&\phi^k\partial_{x^k}(\sigma^i_{\alpha})-\sigma_{\alpha}^k\partial_{x^k}(\phi^i)+C^{\beta}_{\alpha}\sigma^i_{\alpha}+\frac{1}{2}a \sigma^i_{\alpha}=0. &
\label{equation_determining2}
\end{eqnarray}
It is easy to check, by using \refeqn{equation_coeffSDE},  that equation \refeqn{equation_determining1} coincides with equation \refeqn{equation_symmetry3}
with $Y=\tilde{Y}+at\partial_t$ and that, being $C$ antisymmetric,  equation \refeqn{equation_determining2} implies equation \refeqn{equation_symmetry1}.
Furthermore, it is possible to prove that, if $A^{ij}$ has constant rank, there exists a unique antisymmetric matrix $C(x)$ such that, if $Y=\tilde{Y}+at\partial_t$
solves equation \refeqn{equation_symmetry2},  then $(\tilde{Y},C,a)$ solves equation \refeqn{equation_determining2}.
Therefore, providing $A$ is non-degenerate, the symmetries of  a SDE with deterministic time change defined in \cite{DMU1} are in one-to-one correspondence
 with the symmetries of the related martingale problem  introduced here.
The presence of the matrix $C \not =0 $ is essential for the validity of this correspondence. Indeed, since in the martingale problem formulation the Brownian motion
is not  fixed, freezing   the Brownian motion  in the SDE formulation  by choosing $C = 0$ may cause the loss of  some
Lie symmetries (see \cite{Gaeta2000} and \cite{DMU1} for further details).

\section*{Acknowledgements}

The authors would like to thank Prof. Gaeta for his  useful comments and suggestions  in the first part of the work. This work was supported by Gruppo Nazionale Fisica Matematica (GNFM-INdAM).

\bibliographystyle{plain}

\bibliography{SymmDiff(2)}

\def\cprime{$'$}
\begin{thebibliography}{10}

\bibitem{DeLara1995}
Michel Cohen~de Lara.
\newblock Geometric and symmetry properties of a nondegenerate diffusion
  process.
\newblock {\em Ann. Probab.}, 23(4):1557--1604, 1995.

\bibitem{DMU2}
Francesco~C. De~Vecchi, Paola Morando, and Stefania Ugolini.
\newblock Reduction and reconstruction of stochastic differential equations via
  symmetries.
\newblock {\em J. Math. Phys.}, 57(12):123508, 22, 2016.

\bibitem{DMU1}
Francesco~C. De~Vecchi, Paola Morando, and Stefania Ugolini.
\newblock Symmetries of stochastic differential equations: a geometric
  approach.
\newblock {\em J. Math. Phys.}, 57(6):063504, 17, 2016.

\bibitem{Elworthy1982}
K.~David Elworthy.
\newblock {\em Stochastic differential equations on manifolds}, volume~70 of
  {\em London Mathematical Society Lecture Note Series}.
\newblock Cambridge University Press, Cambridge-New York, 1982.

\bibitem{Emery1989}
Michel {\'E}mery.
\newblock {\em Stochastic calculus in manifolds: With an appendix by P.-A.
  Meyer}.
\newblock Universitext. Springer-Verlag, Berlin, 1989.

\bibitem{Fredericks2007}
Ebrahim Fredericks and Fazal~Mahmood Mahomed.
\newblock Symmetries of first-order stochastic ordinary differential equations
  revisited.
\newblock {\em Math. Methods Appl. Sci.}, 30(16):2013--2025, 2007.

\bibitem{Gaeta2000}
Giuseppe Gaeta.
\newblock Lie-point symmetries and stochastic differential equations. {II}.
\newblock {\em J. Phys. A}, 33(27):4883--4902, 2000.

\bibitem{Gaeta1999}
Giuseppe Gaeta and Niurka~Rodr{\'{\i}}guez Quintero.
\newblock Lie-point symmetries and stochastic differential equations.
\newblock {\em J. Phys. A}, 32(48):8485--8505, 1999.

\bibitem{Glover1990}
Joseph Glover and Joanna Mitro.
\newblock Symmetries and functions of {M}arkov processes.
\newblock {\em Ann. Probab.}, 18(2):655--668, 1990.

\bibitem{Hsu2002}
Elton~P. Hsu.
\newblock {\em Stochastic analysis on manifolds}, volume~38 of {\em Graduate
  Studies in Mathematics}.
\newblock American Mathematical Society, Providence, RI, 2002.

\bibitem{Watanabe1981}
Nobuyuki Ikeda and Shinzo Watanabe.
\newblock {\em Stochastic differential equations and diffusion processes},
  volume~24 of {\em North-Holland Mathematical Library}.
\newblock North-Holland Publishing Co., Amsterdam-New York; Kodansha, Ltd.,
  Tokyo, 1981.

\bibitem{Kozlov2010}
Roman Kozlov.
\newblock Symmetries of systems of stochastic differential equations with
  diffusion matrices of full rank.
\newblock {\em J. Phys. A}, 43(24):245201, 16, 2010.

\bibitem{Cami2009}
Joan-Andreu L{\'a}zaro-Cam{\'{\i}} and Juan-Pablo Ortega.
\newblock Reduction, reconstruction, and skew-product decomposition of
  symmetric stochastic differential equations.
\newblock {\em Stoch. Dyn.}, 9(1):1--46, 2009.

\bibitem{Zambrini2004}
Paul Lescot and Jean-Claude Zambrini.
\newblock Isovectors for the {H}amilton-{J}acobi-{B}ellman equation, formal
  stochastic differentials and first integrals in {E}uclidean quantum
  mechanics.
\newblock In {\em Seminar on {S}tochastic {A}nalysis, {R}andom {F}ields and
  {A}pplications {IV}}, volume~58 of {\em Progr. Probab.}, pages 187--202.
  Birkh\"auser, Basel, 2004.

\bibitem{Liao1992}
Ming Liao.
\newblock Symmetry groups of {M}arkov processes.
\newblock {\em Ann. Probab.}, 20(2):563--578, 1992.

\bibitem{Meyer1981}
Paul-Andr\'e Meyer.
\newblock G\'eom\'etrie stochastique sans larmes.
\newblock In {\em Seminar on {P}robability, {XV} ({U}niv. {S}trasbourg,
  {S}trasbourg, 1979/1980) ({F}rench)}, volume 850 of {\em Lecture Notes in
  Math.}, pages 44--102. Springer, Berlin-New York, 1981.

\bibitem{Olver1993}
Peter~J. Olver.
\newblock {\em Applications of {L}ie groups to differential equations}, volume
  107 of {\em Graduate Texts in Mathematics}.
\newblock Springer-Verlag, New York, second edition, 1993.

\bibitem{Williams1987}
L.~Chris~G. Rogers and David Williams.
\newblock {\em Diffusions, {M}arkov processes, and martingales. {V}ol. 2}.
\newblock Wiley Series in Probability and Mathematical Statistics: Probability
  and Mathematical Statistics. John Wiley \& Sons, Inc., New York, 1987.
\newblock It{\^o} calculus.

\bibitem{Schwartz1982}
Laurent Schwartz.
\newblock G\'eom\'etrie diff\'erentielle du 2\`eme ordre, semi-martingales et
  \'equations diff\'erentielles stochastiques sur une vari\'et\'e
  diff\'erentielle.
\newblock In {\em Seminar on {P}robability, {XVI}, {S}upplement}, volume 921 of
  {\em Lecture Notes in Math.}, pages 1--148. Springer, Berlin-New York, 1982.

\bibitem{Srihirun2007}
Boonlert Srihirun, Sergey~V. Meleshko, and Eckart Schulz.
\newblock On the definition of an admitted {L}ie group for stochastic
  differential equations.
\newblock {\em Commun. Nonlinear Sci. Numer. Simul.}, 12(8):1379--1389, 2007.

\bibitem{Stephani1989}
Hans Stephani.
\newblock {\em Differential equations: Their solution using symmetries}.
\newblock Cambridge University Press, Cambridge, 1989.

\bibitem{Stroock1979}
Daniel~W. Stroock and S.~R.~Srinivasa Varadhan.
\newblock {\em Multidimensional diffusion processes}, volume 233 of {\em
  Grundlehren der Mathematischen Wissenschaften [Fundamental Principles of
  Mathematical Sciences]}.
\newblock Springer-Verlag, Berlin-New York, 1979.

\bibitem{Unal2004}
Gazanfer {\"U}nal and Jian-Qiao Sun.
\newblock Symmetries and conserved quantities of stochastic dynamical control
  systems.
\newblock {\em Nonlinear Dynam.}, 36(1):107--122, 2004.

\end{thebibliography}

\end{document}